\documentclass[12pt,twoside,reqno]{amsart}
\usepackage[colorlinks=true,citecolor=blue]{hyperref}
\usepackage{mathptmx, amsmath, amssymb, amsfonts, amsthm, mathptmx, enumerate, color}
\usepackage{graphicx}
\usepackage{epstopdf}

\newtheorem{theorem}{Theorem}[section]

\newtheorem{lemma}{Lemma}[section]
\newtheorem{proposition}{Proposition}[section]
\theoremstyle{definition}
\newtheorem{definition}{Definition}[section]
\newtheorem{example}{Example}[section]

\numberwithin{equation}{section}
\newtheorem{condition}[theorem]{Condition}
\newtheorem{algorithm}[theorem]{Algorithm}
\newtheorem{problem}[theorem]{Problem}
\begin{document}
\setcounter{page}{1}

\vspace*{1.0cm}
\title[Data-compatibility of algorithms]
{Data-compatibility of algorithms}
\author[Y. Censor, M. Zaknoon, A. J. Zaslavski]{ Yair Censor$^{1,*}$, Maroun Zaknoon$^2$, Alexander J. Zaslavski$^3$}
\maketitle
\vspace*{-0.6cm}

\begin{center}
{\footnotesize {\it

$^1$Department of Mathematics, University of Haifa\\Mt. Carmel, Haifa 3498838, Israel\\
$^2$Department of Mathematics, The Arab Academic College for Education\\22 HaHashmal Street, Haifa 32623, Israel\\
$^3$Department of Mathematics, The Technion -- Israel Institute of Technology\\Technion City, Haifa 3200003, Israel

}}\end{center}

\vskip 4mm {\small \noindent {\bf Abstract.}
The data-compatibility approach to constrained optimization, proposed here,
strives to a point that is \textquotedblleft close enough\textquotedblright%
\ to the solution set and whose target function value is \textquotedblleft
close enough\textquotedblright\ to the constrained minimum value. These
notions can replace analysis of asymptotic convergence to a solution point of
infinite sequences generated by specific algorithms. We consider a problem of
minimizing a convex function over the intersection of the fixed point sets of
nonexpansive mappings and demostrate the data-compatibility of the Hybrid
Subgradient Method (HSM). A string-averaging HSM is obtained as a by-product
and relevance to the minimization over disjoint hard and soft constraints sets
is discussed.

 }

\renewcommand{\thefootnote}{}
\footnotetext{ $^*$Corresponding author.
\par
E-mail addresses: yair@math.haifa.ac.il (Y. Censor), zaknoon@arabcol.ac.il (M. Zaknoon),\\ ajzasl@techunix.technion.ac.il (A. J. Zaslavski).
\par
Received May 17, 2020; Accepted October 19, 2020. }

\section{Introduction}

The data of a constrained minimization problem $\min\{f(x)\mid$ $x\in C\}$
consists of a target function $f$ and a constraints set $C.$ For this problem
to be meaningful, $C$ needs to be nonempty, and for asymptotic convergence
analysis of an algorithm for solving the problem one commonly needs that the
solution set of the problem be nonempty, i.e., that there exists at least one
point, say $x^{\ast},$ in $C$ with the property that $f(x^{\ast})\leq f(x)$
for all $x\in C.$

In real-world practical situations these nonemptiness assumptions cannot
always be guaranteed or verified. To cope with this we define the notion of
\textit{data-compatibility} in a Hilbert space. Such data-compatibility is a
finite, not an asymptotic, notion. Even when the sets $C$ and the solution set
of the constrained minimization problem are nonempty, striving for
data-compatibility is a worthwhile aim because it can be \textquotedblleft
reached\textquotedblright,\ contrary to asymptotic limit points.

Data-compatibility of a point $x$ means that it simultaneously, (i) is
\textquotedblleft close enough\textquotedblright\ to the set of minimizers of
$f$ over the constraints $C,$ and (ii) has a function value $f(x)$
\textquotedblleft close enough\textquotedblright\ to the minimum of $f$ over
that set of constraints. Once we precisely define these notions the question
arises whether or not it is possible to guarantee that, under certain
conditions, compatibility with a data pair $(C,f)$ can be reached by an
iterative process designed to solve a constrained optimization problem?

The advantage of these data-compatibility\ notions is that they can cater
better to practical situations. On the theoretical side, we propose that
instead of proving asymptotic convergence of iterative processes and
afterwards studying the effects of various stopping rules, it is possible to
directly formulate data-compatibility and provably guarantee that it can be
reached by an iterative process.

We demonstrate our principle and approach in a specific scenario. The problem
formulation and the algorithm that we use in our demonstration serve only as
vehicles to present our data-compatibility approach which we believe is novel.
To do so we study the behavior of the specific iterative process for convex
minimization over the intersection of the fixed point sets of nonexpansive
operators, called the Hybrid Subgradient Method (HSM)\footnote{The term HSM is
in analogy with the established term, coined by Yamada \cite{Yamada2001a}, of
the Hybrid Steepest Descent Method (HSDM). The structural similarity of the
HSM with the HSDM is that the former uses subgradient steps instead of the
steepest descent steps used by the latter.}. In contrast with many existing
works, rather than investigating asymptotic convergence of the generated
sequences we specify conditions under which the iterative HSM process
generates solutions that are compatible with the data pair $(\mathrm{Fix}%
\left(  T\right)  ,f).$

Minimization over the intersection of the fixed point sets of nonexpansive
operators has been treated extensively in the literature, of which we
reference a few works below. But in all these earlier works the asymptotic
convergence of algorithms, under various sets of conditions, is the central
theme, not data-compatibility.

As an important special case of the general algorithmic formulation we discuss
a string-averaging algorithmic scheme. The string-averaging algorithmic notion
has a quite general structure in itself. Invented in \cite{ceh01} and spurred
many extensions and applications since then, e.g., \cite{Bargetz2018,Kong2019}
and the book \cite{Zaslavskibook2018}, it works in general as follows. From a
current iteration point, it performs consecutively specified iterative
algorithmic steps \textquotedblleft along\textquotedblright\ different
\textquotedblleft strings\textquotedblright\ of individual constraints sets
and then takes a combination, convex or other, of the strings' end-points as
the next iterate. The string-averaging algorithmic scheme gives rise to a
variety of specific algorithms by judiciously choosing the number of strings,
their assignments and the nature of the combination of the strings'
end-points. Details are given in the sequel.

Earlier works on minimizing convex functions over the intersection of the
fixed point sets of nonexpansive mappings are all based on asymptotic
convergence of the algorithms and investigate the problem and prove
convergence of algorithms under various conditions. These include, to name but
a few, the papers of Iiduka \cite{Iiduka2012}, \cite{Iiduka2015a},
\cite{Iiduk2016}, the work of Maing\'{e} \cite{Mainge2008}, and publications
by Hayashi and Iiduka \cite{Hayashi2018} and Deutsch and Yamada
\cite{Deutsch1998}. Methods similar to HSM were studied by several
researchers, e.g., Shor \cite{Shor1985}, Albert, Iusem and Solodov
\cite{Albert1998}, Yamada and Ogura \cite{yamada2005hybrid}, Hirstoaga
\cite{Hirstoaga2006}, Martinez-Yanes and Xu \cite{Martinez-Yanes2006},
Maing\'{e} \cite{Mainge2008a}, Aoyama and Kohsaka \cite{Aoyama2014}, Cegielski
\cite{Cegielski2015}, Bello Cruz \cite{Cruz2017}.

In contrast with these, our main contribution is in the novel quest for
data-compatibility instead of asymptotic convergence properties.

The paper is organized as follows. In Section \ref{sec:data-comp} we define
the notion of data-compatibility of a point with the data of a constrained
minimization problem, and, in particular in Subsection
\ref{subsec:investigating} we discuss the problem of guaranteeing a priori
data-compatibility. In Section \ref{sect:origin} we present the origin and
motivation of our work. The problem of minimizing a convex function over the
intersection of the fixed point sets of nonexpansive mappings is defined in
Section \ref{sec:Assump present prob} along with the Hybrid Subgradient Method
(HSM) for its solution. Inexact iterates are discussed in Section
\ref{Inexact Iter section} followed in Section \ref{sec:convergence} by work
on the main result that proves the ability of the HSM to generate a
data-compatible point for the problem. We present the string-averaging variant
of the HSM in Section \ref{sect:SAPv}. In Section \ref{sect:inconsistent} we
consider a specific situation wherein the data of a constrained minimization
problem does not necessarily obey feasibility of the constraints, i.e., does
not demand that $C=\cap_{i=1}^{m}C_{i}$ is nonempty. Finally, in Section
\ref{sect:hard-soft} we describe the minimization over disjoint hard and soft
constraints sets problem and its relation to the work presented in this paper.

\section{Data-compatibility \label{sec:data-comp}}

In this section we define the notion of \textit{data-compatibility of a point
with the data of a constrained minimization problem. }Let $\Omega\subseteq H$
be a given nonempty set in the Hilbert space $H$ and let there be given, for
$i=1,2,\ldots,m,$ nonempty sets $C_{i}\subseteq\Omega.$ We denote by
$\Gamma:=\{C_{i}\}_{i=1}^{m}$ the family of sets and refer to it as the
\textquotedblleft constraints data $\Gamma$\textquotedblright.

We introduce a set $\Delta$ such that $\Omega\subseteq\Delta\subseteq H$ and
assume that we are given a function $f:\Delta\rightarrow R$ which is referred
to as \textquotedblleft the target function $f$\textquotedblright\ or, in
short, the data $f$. A pair $(\Gamma,f)$ is referred to as the
\textquotedblleft data pair $(\Gamma,f)$\textquotedblright.

\subsection{Data-compatibility for constraints\label{subsec:D-C-const}}

Constraint modelling has a prominent role in operations research and is used
in a wide range of industrial projects, such as, but by far not only, the
control of an intelligent interface linking computer aided design and
automatic inspection systems, the identification of manufacturing errors from
inspection results and the design synthesis and analysis of mechanisms, to
name a few, see, e.g., \cite{Hicks2006}. In our language, it is the modelling
of real-world problems via \textit{convex feasibility problems} (CFPs). Given
a finite family of, commonly convex, sets $\Gamma:=\{C_{i}\}_{i=1}^{m}$ the
CFP is to find a point $x^{\ast}\in C:=\cap_{i=1}^{m}C_{i}.$ This approach has
a long history, see, e.g., Combettes' or Bauschke and Borwein's corner stone
reviews \cite{Combettes1993}, \cite{Bauschke96}, respectively, Cegielski's
book \cite{Cegielski2012Book} and Bauschke and Combettes' book \cite{BC11}.
First we look at compatibility with the constraints data alone. For this we
need an appropriate \textit{proximity function} that \textquotedblleft
measures\textquotedblright\ how incompatible an $x\in\Omega$ is with the
constraints of $\Gamma$. There is no common definition in the literature, but
a proximity function $\operatorname*{Prox}_{\Gamma}:\Omega\rightarrow R_{+}$
(the nonnegative orthant) should have the property that $\operatorname*{Prox}%
_{\Gamma}(x)=0$ if and only if $x\in C:=\cap_{i=1}^{m}C_{i}.$ By
evaluating/measuring a \textquotedblleft distance\textquotedblright\ to the
constraints, the lower the value of $\operatorname*{Prox}_{\Gamma}(x)$ is --
the less incompatible $x$ is with the constraints.

A proximity function does not require that $C\neq\varnothing$ and it is a
useful tool, particularly for situations when $C\neq\varnothing$ does not
hold, or cannot be verified. An enlightening discussion of proximity functions
for the convex feasibility problem can be found in Cegielski's book
\cite[Subsection 1.3.4]{Cegielski2012Book}. An important and often used
proximity function is%
\begin{equation}
\text{Prox}_{\Gamma}(x):=\frac{1}{2}\sum_{i=1}^{m}w_{i}\left\Vert P_{C_{i}%
}(x)-x\right\Vert ^{2}, \label{eq:prox-function}%
\end{equation}
where $P_{C_{i}}(x)$ is the orthogonal (metric) projection onto $C_{i}$ and
$\{w_{i}\}_{i=1}^{m}$ is a set of weights such that $w_{i}\geq0$ and
$\sum_{i=1}^{m}w_{i}=1.$ With a proximity function at hand we define
compatibility with constraints as follows.

\begin{definition}
\label{def:epsilon-comp}$\gamma$\textbf{-compatibility with constraints data
}$\Gamma$\textbf{. }Given constraints data $\Gamma,$ a proximity function
$\operatorname*{Prox}_{\Gamma},$ and a real number $\gamma\geq0,$ we say that
a point $x\in\Omega$ is \textquotedblleft$\gamma$-compatible with $\Gamma
$\textquotedblright\ if $\operatorname*{Prox}_{\Gamma}(x)\leq\gamma.$ We
define the set of all points that are $\gamma$-compatible with $\Gamma$ by
$\Pi(\Gamma,\gamma):=\{x\in\Omega\mid\operatorname*{Prox}_{\Gamma}%
(x)\leq\gamma\}$ and call it the $(\Gamma,\gamma)$-compatibility set.
\end{definition}

\begin{definition}
\label{def:gamma-output}\textbf{The} $\gamma$\textbf{-output of a sequence.
}Given constraints data $\Gamma,$ a proximity function $\operatorname*{Prox}%
_{\Gamma},$ a real number $\gamma\geq0$ and a sequence $R:=\{x^{k}%
\}_{k=0}^{\infty}$ of points in $\Omega$, we use $O\left(  \Gamma
,\gamma,R\right)  $ to denote the point $x\in\Omega$ that has the following
properties: $\operatorname*{Prox}_{\Gamma}(x)\leq\gamma,$ and there is a
nonnegative integer $K$ such that $x^{K}=x$ and, for all nonnegative integers
$k<K$, $\operatorname*{Prox}_{\Gamma}(x)>\gamma$. If there is such an $x$,
then it is unique. If there is no such $x$, then we say that $O\left(
\Gamma,\gamma,R\right)  $ is \textit{undefined}, otherwise it is
\textit{defined}.
\end{definition}

If $R$ is an infinite sequence generated by a certain process then $O\left(
\Gamma,\gamma,R\right)  $ is the \textit{output} produced by that process when
we add to it instructions that make it terminate as soon as it reaches a point
that is $\gamma$-compatible with $\Gamma$.

The $(\Gamma,\gamma)$-compatibility set $\Pi(\Gamma,\gamma)\subseteq\Omega$
need not be nonempty for all $\gamma.$ If, however, $\Pi(\Gamma,0)\neq
\varnothing$ then $\Pi(\Gamma,0)=C.$ We have used the notion of $\gamma
$-compatibility with constraints earlier in our work on the superiorization
method, see, e.g., \cite{Censor2019}. The unconstrained minimization of
Prox$_{\Gamma}(x)$ always yields an infimum of Prox$_{\Gamma}(x)$ but this
does not guarantee that $\Pi(\Gamma,\gamma)$ is nonempty. However, $\Pi
(\Gamma,\gamma)$ can be nonempty even if the constraints intersection
$C=\cap_{i=1}^{m}C_{i}$ is empty. In order to proceed with data-compatibility
for constrained minimization in the next subsection we require that
$\Pi(\Gamma,\gamma)\neq\varnothing.$

It should be emphasized that a sequence considered in Definition
\ref{def:gamma-output} need not be convergent, or can be convergent but not
necessarily to a point in $\Pi(\Gamma,\gamma)$, and still yield an output
$O\left(  \Gamma,\gamma,R\right)  $ that is $\gamma$-compatible with $\Gamma.$

\subsection{Data-compatibility for constrained
minimization\label{subsec:D-C-const-minim}}

For a $\gamma,$ for which $\Pi(\Gamma,\gamma)\neq\varnothing,$ we define the
set of minimizers of $f$ over $\Pi(\Gamma,\gamma),$
\begin{equation}
S(f,\Pi(\Gamma,\gamma)):=\{x\in\Pi(\Gamma,\gamma)\mid f(x)\leq f(y),\text{ for
all }y\in\Pi(\Gamma,\gamma)\}.
\end{equation}
If $f$ is the zero function or if $f=$constant then $S(f,\Pi(\Gamma
,\gamma))=\Pi(\Gamma,\gamma).$ We use the distance function between a point
$x$ and a set $S$ defined as%
\begin{equation}
d(x,S):=\inf\{d(x,y)\ |\ y\in S\}
\end{equation}
where $d(x,y)$ is the distance between points $x$ and $y.$ Next we propose our
definition of compatibility with a data pair $(\Gamma,f)$.

\begin{definition}
\label{def:tau-el-compatibility}$(\tau,\bar{L})$\textbf{-compatibility with a
data pair }$(\Gamma,f).$ Given a $\tau\geq0,$ and a real number $\bar{L}>0,$
we say that a point $x\in\Omega$ is \textquotedblleft$(\tau,\bar{L}%
)$-compatible with the data pair $(\Gamma,f)$\textquotedblright\ if
$S(f,\Pi(\Gamma,\gamma))\neq\varnothing$ and the following two conditions hold%
\begin{gather}
d(x,S(f,\Pi(\Gamma,\gamma)))\leq\tau\text{ \label{eq:comp-delta}}\\
\text{and }\nonumber\\
f(x)\leq f(z)+\tau\bar{L},\text{ for all }z\in S(f,\Pi(\Gamma,\gamma)),
\label{eq:comp-S}%
\end{gather}
where the constraints data $\Gamma,$ a proximity function
$\operatorname*{Prox}_{\Gamma},$ a target function $f,$ and a $\gamma\geq0$
such that $\Pi(\Gamma,\gamma)\neq\varnothing,$ are given.
\end{definition}

This means that such a point $x$ simultaneously, (i) is \textquotedblleft
close enough\textquotedblright\ to the set of minimizers of $f$ over a
$(\Gamma,\gamma)$-compatibility set of the constraints, and (ii) has a
function value $f(x)$ \textquotedblleft close enough\textquotedblright\ to the
minimum of $f$ over that $(\Gamma,\gamma)$-compatibility set of the
constraints. This definition does not require nonemptiness of the intersection
of the constraints $C=\cap_{i=1}^{m}C_{i}$ neither does it require that the
constrained minimization problem $\min\{f(x)\mid$ $x\in C\},$ has a nonempty
solution set $\operatorname*{SOL}(f,C)$ which is defined by%
\begin{equation}
\operatorname*{SOL}(f,C):=S(f,C)=\{x\in C\mid f(x)\leq f(y),\text{ for all
}y\in C\}. \label{eq:SOL}%
\end{equation}
It relies on the weaker assumptions that $\Pi(\Gamma,\gamma)\neq\varnothing$
and $S(f,\Pi(\Gamma,\gamma))\neq\varnothing.$ Therefore, use of these notions
makes it possible to deviate from the nonemptiness assumptions which usually
lie at the heart of asymptotic analyses in optimization theory.

\begin{example}
Here is a motivating example from a specific real-world application that
stands to benefit from the situation described above. A fully-discretized
modeling approach to intensity-modulated radiation therapy (IMRT) treatment
planning, e.g., \cite{Censor2008}, leads to a very large and sparse system of
linear inequalities, see, e.g., \cite[Equation 6.13]{Brooke-2019} which is, in
practice, further equipped with box constraints on the unknown vector $x$
there. The sets $C_{i}$ are half-spaces and the, commonly used, proximity
function (\ref{eq:prox-function}) above always has a minimum value $g$ and its
solution set is nonempty. This is guaranteed, e.g., by Proposition 7 of
\cite{Combettes1994} with $\Omega=H$ and the box constraints serving the
condition there that one of the sets must be bounded. Therefore, $\Pi
(\Gamma,\gamma)$ will be nonempty for any $\gamma$ for which $g\leq\gamma.$ 

Often, an exogeneous target function is imposed on these constraints in order
to impose some prior knowledge. A commonly employed such function $\ f$ is the
\textquotedblleft total variation\textquotedblright\ (TV), see, e.g.,
\cite{TV2011} that smoothes the solution vector, see, e.g., \cite{DCSGX2015}.
In this situation we know in advance that $S(f,\Pi(\Gamma,\gamma
))\neq\emptyset$ without having to invest preliminary computing resources in
finding whether $C=\cap_{i=1}^{m}C_{i}$ is or is not empty. This example is
also relevant to Theorems \ref{thm:7.1 string} and \ref{thm:7.1 Proj} below.
\end{example}

In the consistent case when $C\neq\varnothing$ and $\operatorname*{SOL}%
(f,C)\neq\varnothing$, i.e., $\gamma=0$ and $\Pi(\Gamma,0)=C,$ Definition
\ref{def:tau-el-compatibility} takes the following form.

\begin{definition}
\label{def:consistent-data-comp}$(\tau,\bar{L})$\textbf{-compatibility with a
data pair }$(\Gamma,f)$\textbf{ in the consistent case}$.$ Given consistent
constraints data $\Gamma$ via $C:=\cap_{i=1}^{m}C_{i}\neq\varnothing,$ a
target function $f,$ a $\tau\geq0,$ and a real number $\bar{L}>0,$ we say that
a point $x\in\Omega$ is \textquotedblleft$(\tau,\bar{L})$-compatible with the
consistent data pair $(\Gamma,f)$\textquotedblright\ if $\operatorname*{SOL}%
(f,C)\neq\varnothing$ and the following two conditions hold%
\begin{gather}
d(x,\operatorname*{SOL}(f,C))\leq\tau\text{ \label{eq:data-comp-1}}\\
\text{and }\nonumber\\
f(x)\leq f(z)+\tau\bar{L},\text{ for all }z\in\operatorname*{SOL}%
(f,C).\text{\label{eq:data-comp-2}}%
\end{gather}

\end{definition}

\subsection{Data-compatibility instead of asymptotic
convergence\label{subsec:investigating}}

We consider in this paper the consistent case situation of Definition
\ref{def:consistent-data-comp}. Given a constrained minimization problem
$\min\{f(x)\mid$ $x\in C\}$ via its data pair $(\Gamma,f),$ one traditional
route in optimization theory is to design/construct an iterative process for
its solution and investigate the asymptotic convergence of the this process.
After asymptotic convergence of the iterative process has been secured a
common approach is to use the process as a basis for creating an
algorithm\footnote{An algorithm is a finite sequence of well-defined,
computer-implementable instructions, typically to solve a class of problems or
to perform a computation. However, it is common practice in the literature to
use the term \textquotedblleft algorithm\textquotedblright\ also for iterative
processes that do not include a stopping rule.}. Such an algorithm will use
the iteration formulae dictated by the process but have a user-chosen stopping
rule attached to it. The, so obtained, algorithm will stop when the stopping
rule is met and output an approximate solution to the original problem.

Various stopping rules are in use and we do not attempt to review them. But it
is clear that the question \textquotedblleft when to stop\textquotedblright%
\ has a\ profound influence on the practical output of an algorithmic run.
With roots and research in the fields of statistics and optimization,
\textquotedblleft optimal stopping is concerned with the problem of choosing a
time to take a given action based on sequentially observed random variables in
order to maximize an expected payoff or to minimize an expected
cost.\textquotedblright, see \cite{opt-stop-book-2008}.

One aspect of stopping rules is the question whether or not a particular
stopping rule will actually cause the iterative process to which it is
attached to stop and yield an output. If, for example, the stopping rule is to
stop after a specified number of iterations then the algorithmic run will
undoubtedly stop when this number is reached. On the other hand, if one
considers a feasibility-seeking problem to find a point in the intersection of
a given finite family of sets $C:=\cap_{i=1}^{m}C_{i}$ then a
feasibility-seeking iterative process that is proven to asymptotically
converge to a point in $C$ can be used. However, if one uses the stopping rule
that iterations should stop when an iterate $x^{k}$ is reached that belongs to
the intersection $C$ then the process might never stop.

\begin{problem}
The question that we ask ourselves here is whether or not it is possible to a
priori guarantee that, under certain conditions, $(\tau,\bar{L})$%
-compatibility with a data pair $(\Gamma,f)$ can be reached by an iterative
process designed to solve a constrained optimization problem?
\end{problem}

Obviously, if the considered iterative process is proven to asymptotically
converge to a point in $\operatorname*{SOL}(f,C))$ then the answer is positive
and (\ref{eq:comp-delta})-(\ref{eq:comp-S}) can be used as a stopping rule
that will indeed yield a data-compatible output. On the other hand, if an
algorithm generates $(\tau,\bar{L})$-compatible sequences for all values of
$\tau$ then all the sequences generated by it are asymptotically convergent.
Nevertheless, the notion of $(\tau,\overline{L})$-compatibility makes sense
even if it holds for certain (not all) values of $\tau$ and in this case it
can provide useful information. To approach this, inspired by \cite[Definition
2.1]{Censor2019}, we suggest the following definition for the first iterate
that satisfies both conditions (\ref{eq:data-comp-1})-(\ref{eq:data-comp-2}).

\begin{definition}
\textbf{The }$(\tau,\bar{L})$\textbf{-output of a sequence}. Given constraints
data $\Gamma,$ a proximity function $\operatorname*{Prox}_{\Gamma},$ a target
function $f,$ a $\gamma\geq0,$ a $\tau\geq0,$ and a real number $\bar{L}>0,$
we consider an infinite sequence $\mathcal{X}=\{x^{k}\}_{k=0}^{\infty}$ of
points in $\Omega.$ Let $\operatorname*{OUT}((\Gamma,f),\gamma,(\tau,\bar
{L}),\mathcal{X})$ denote the point $x\in\Omega$ \ that fulfills
(\ref{eq:data-comp-1})-(\ref{eq:data-comp-2}) and such that there exists a
nonnegative integer $K$ such that $x^{K}=x$ and for all nonnegative integers
$k<K$ at least one of the two conditions (\ref{eq:comp-delta}%
)-(\ref{eq:comp-S}) is violated. If there is such an $x,$ then it is unique.
If there is no such $x$ then we say that $\operatorname*{OUT}((\Gamma
,f),\gamma,(\tau,\bar{L}),\mathcal{X})$ is undefined, otherwise it is defined.
\end{definition}

If $\mathcal{X}$ is generated by an iterative process, then
$\operatorname*{OUT}((\Gamma,f),\gamma,(\tau,\bar{L}),\mathcal{X})$ is the
output produced by that process when we add to it instructions that make it
terminate as soon as it reaches a point that is $(\tau,\bar{L})$-compatible
with a data pair $(\Gamma,f)$.

General results that will characterize, or give conditions for, an iterative
process to be provably data-compatible are not yet known, but to initiate
research in this direction we demonstrate our approach in a specific scenario.

We work in the framework of Definition \ref{def:consistent-data-comp} and
study the behavior of an iterative process for convex minimization over fixed
point sets of nonexpansive operators. Rather than generating infinite
sequences that asymptotically converge to a point in $\operatorname*{SOL}%
(f,C),$ we specify conditions under which the iterative process generates
solutions that are $(\tau,\bar{L})$-compatible with the data pair
$(\Gamma,f).$

In the sequel (Section \ref{sect:inconsistent}) we also discuss a specific
situation wherein the data pair $(\Gamma,f),$ with $\Gamma:=\{C_{i}%
\}_{i=1}^{m}$ a family of closed and convex subsets of $H,$ not necessarily
obeys that $C=\cap_{i=1}^{m}C_{i}$ is nonempty.

\section{Origin and motivation\label{sect:origin}}

The origin of the idea of data-compatibility is the work done in
\cite{Censor2014}. There the string-averaging projected subgradient method
(SA-PSM) was developed and studied. The SA-PSM is a variant of the projected
subgradient method for solving constrained minimization problems. It differs
from the traditional projected subgradient method (PSM) by replacing, in each
iteration, the single step of projecting onto the entire constraints set with
projections onto the individual sets whose intersection is the feasible set of
the minimization problem. This is an advantage, in the frequently occurring
situations, when the individuals sets are easier to project on than their
entire intersection.

The main result in \cite[Theorem 9]{Censor2014} is not really an asymptotic
convergence theorem. Instead, it provably guarantees the ability of the
algorithm to reach an iterate of the generated sequence that is, up to
predetermined bounds, close to a \textquotedblleft solution\textquotedblright.
Here we formulate this notion and extend it to encompass also the case when
the solution set of the problem is empty.

\section{The problem and the algorithm \label{sec:Assump present prob}}

Our problem, iterative process, and main result are set in a Hilbert space.
However, some intermediate auxiliary results are true, and have some
independent value, in a general metric space. Let $\left(  X,\rho\right)
\ $be a complete metric space and let $T:X\longrightarrow X$ be an operator.
The fixed point set\ of $T$\ is%
\begin{equation}
\mathrm{Fix}\left(  T\right)  :=\left\{  x\in X\ \mid T(x)=x\right\}  .
\label{eq:1.9}%
\end{equation}

An operator $T$ is nonexpansive if it satisfies%
\begin{equation}
\rho\left(  T(x),T(y)\right)  \leq\rho\left(  x,y\right)  ,\text{ for all
}x,y\in X. \label{eq:1.8}%
\end{equation}

Given a nonempty set $E\subseteq X$ define the distance of a point $x\in X$
from it by%
\begin{equation}
d(x,E):=\inf\{\rho(x,y)\ |\ y\in E\}. \label{eq:1.8 m 1}%
\end{equation}

We denote\ the ball with center at a given $x\in X$ and radius $r>0$ by
$B(x,r).$ The execution of the operator $T$ for $n$ times consecutively on an
initial given point $x$ is denoted by $T^{n}x,$ and $T^{0}x:=x.$

For $X=H$ a Hilbert space, we look at a constrained minimization problem of
the form%
\begin{equation}
\min\{f(x)\mid x\in\mathrm{Fix}\left(  T\right)  \} \label{prob:cons-min-1}%
\end{equation}
where $f$ is a convex target function from $H$ into the reals and $T$ is a
given nonexpansive operator. Solving this problem means to%
\begin{equation}
\text{find a point }x\text{ in }\operatorname*{SOL}(f,\mathrm{Fix}\left(
T\right)  ), \label{prob:cons-min Version 2}%
\end{equation}
where
\begin{equation}
\operatorname*{SOL}(f,\mathrm{Fix}\left(  T\right)  ):=\{x\in\mathrm{Fix}%
\left(  T\right)  \mid f(x)\leq f(y){\text{ for all }}y\in\mathrm{Fix}\left(
T\right)  \}. \label{Solutoin Set}%
\end{equation}

For this task we employ an iterative Hybrid Subgradient Method (HSM) that uses
the powers of the operator $T$ combined with subgradient steps. We denote by
$\partial f(x^{k})$ the subgradient set of $f$ at $x^{k}.$

\begin{algorithm}
\label{alg:sa-psm}$\left.  {}\right.  $\textbf{Hybrid Subgradient Method
(HSM).}

\textbf{Initialization}: Let $\{\alpha_{k}\}_{k=0}^{\infty}\subset(0,1]$ be a
scalar sequence and let $x^{0}\in H$ be an arbitrary initialization vector.

\textbf{Iterative step}: Given a current iteration vector $x^{k}$ calculate
the next vector as follows:

Choose any $s^{k}\in\partial f(x^{k})$ and calculate%
\begin{equation}
x^{k+1}=T\left(  x^{k}-\alpha_{k}\frac{s^{k}}{\parallel s^{k}\parallel
}\right)  \text{,} \label{eq:alg-sa-psm-2}%
\end{equation}

but if $s^{k}=0$ then set $\frac{\textstyle s^{k}}{\textstyle\parallel
s^{k}\parallel}:=0.$
\end{algorithm}

Note that if $s^{k}=0$ then the algorithm simply calculates%
\begin{equation}
x^{k+1}=T(x^{k})\text{,} \label{eq:alg-sa-psm-1}%
\end{equation}
otherwise, it uses (\ref{eq:alg-sa-psm-2}).

As mentioned above, our data-compatibility result, presented in Theorem
\ref{thm:7.1} below, will not be about asymptotic convergence but rather
specify conditions that guarantee the existence of an iterate, of any sequence
generated by the HSM of Algorithm \ref{alg:sa-psm}, that is $(\tau,\bar{L}%
)$-compatible with the data pair $(\Gamma=Fix(T),f).$ I.e., that for every
$\tau\in(0,1),$ and any sequence $\{x^{k}\}_{k=0}^{\infty}$, generated by
Algorithm \ref{alg:sa-psm}, there exists an integer $K$ so that, for all
$k\geq K$:%
\begin{gather}
d(x^{k},\operatorname*{SOL}(f,\mathrm{Fix}\left(  T\right)  ))\leq\tau{\text{
}}\\
{\text{and }}\nonumber\\
f(x^{k})\leq f(z)+\tau\bar{L}\text{ for all }z\in\operatorname*{SOL}%
(f,\mathrm{Fix}\left(  T\right)  )
\end{gather}
where $\bar{L}$ is some well-defined constant.

\section{Inexact iterates\label{Inexact Iter section}}

In this sections we establish some properties of sequences of the form
$\left\{  T^{j}y^{0}\right\}  _{j=0}^{\infty},$ for any $y^{0}\in X,$ with
\textquotedblleft computational errors\textquotedblright. These will serve as
tools in proving the main result. In our work we need to focus on operators
that have the property formulated in the next condition.

\begin{condition}
\label{cond:condition}Let $X\ $be a complete metric space, assume that
$T:X\rightarrow X$ is a nonexpansive operator such that $\lim_{j\rightarrow
\infty}T^{j}y^{0}$ exists for any $y^{0}\in X$.
\end{condition}

Condition \ref{cond:condition} holds for many nonexpansive mappings. In
\cite{Reich2014} it was shown that for several important classes of
nonexpansive mappings this property is generic (typical) in the sense of the
Baire category. This means that a class of mappings is equipped with an
appropriate complete metric and it is shown the existence of a subset of the
space of mappings which is a countable intersection of open everywhere dense
sets such that any mappings from the intersection possesses the desirable
convergence property.

\begin{proposition}
\label{The limit is fixed point}Let $X\ $be a complete metric space, and that
$T:X\rightarrow X$ is a nonexpansive operator.

\begin{enumerate}
\item The set $\mathrm{Fix}\left(  T\right)  $ is closed.

\item If $\lim_{j\rightarrow\infty}T^{j}y^{0}$ exists for some $y^{0}\in X$,
then $\lim_{j\rightarrow\infty}T^{j}y^{0}$ is a fixed point of $T$ and,
consequently, $\mathrm{Fix}\left(  T\right)  \not =\varnothing.$
\end{enumerate}
\end{proposition}

\begin{proof}
The proof is obvious.
\end{proof}

\begin{proposition}
\label{Sa Prop 1.3}Let $X\ $be a complete and compact metric space, assume
that $T:X\rightarrow X$ is a nonexpansive operator and let $\varepsilon>0$.
Then there exists a $\delta>0$ such that for each $x\in X$ satisfying
$\rho(x,Tx)\leq\delta$ we have%
\begin{equation}
d(x,\mathrm{Fix}(T))\leq\varepsilon.
\end{equation}

\end{proposition}

\begin{proof}
Assume to the contrary that for each integer $k\geq1$ there exists a point
$x^{k}\in X$ such that%
\begin{equation}
\rho(x^{k},Tx^{k})\leq k^{-1}\text{ and}\;d(x^{k},\mathrm{Fix}(T))>\varepsilon
. \label{Sa (1.2)}%
\end{equation}
Since $X$ is compact, extracting a subsequence and re-indexing, if necessary,
we may assume without loss of generality that, the sequence $\{x^{k}%
\}_{k=1}^{\infty}$ so generated by the repeated use of the above negation,
converges and denote%
\begin{equation}
z:=\lim_{k\rightarrow\infty}x^{k}. \label{Sa (1.3)}%
\end{equation}
Since $T$ is nonexpansive, inequality (\ref{Sa (1.2)}) and the limit
(\ref{Sa (1.3)}) yield, for all integers $k\geq1$,%
\begin{align}
\rho(z,Tz)  &  \leq\rho(z,x^{k})+\rho(x^{k},Tx^{k})+\rho(Tx^{k},Tz)\nonumber\\
&  \leq2\rho(z,x^{k})+k^{-1}\rightarrow0,\text{ as }k\rightarrow\infty,
\end{align}
thus,%
\begin{equation}
z\in Fix(T).
\end{equation}
In view of (\ref{Sa (1.3)}), for all sufficiently large integers $k$,%
\begin{equation}
d(x^{k},\mathrm{Fix}(T))\leq d(x^{k},z)<\varepsilon.
\end{equation}
This contradicts (\ref{Sa (1.2)}) and completes the proof.
\end{proof}

\begin{lemma}
\label{Sa Lem 1.5}Let $X\ $be a complete and compact metric space, assume that
$T:X\rightarrow X$ is a nonexpansive operator for which Condition
\ref{cond:condition} holds, and let $\mu>0$. Then there exists an integer
$k_{1}$ such that for each $x\in X$ there exists $j\in\{0,1,\dots,k_{1}\}$
such that
\begin{equation}
d(T^{j}x,\mathrm{Fix}(T))\leq\mu.
\end{equation}

\end{lemma}

\begin{proof}
Assume to the contrary that for each integer $k\geq1$ there exists a point
$x^{k}\in X$ such that%
\begin{equation}
d(T^{j}x^{k},\mathrm{Fix}(T))>\mu,\text{\ for all }j=0,1,\dots,k.
\label{Sa (1.4)}%
\end{equation}
Since $X$ is compact, extracting a subsequence and re-indexing, if necessary,
we may assume without loss of generality that, the sequence $\{x^{k}%
\}_{k=1}^{\infty}$ so generated by the repeated use of the above negation,
converges and let%
\begin{equation}
z=\lim_{k\rightarrow\infty}x^{k}. \label{Sa (1.5)}%
\end{equation}
By Proposition \ref{The limit is fixed point} and Condition
\ref{cond:condition}, we conclude that%
\begin{equation}
d(T^{j}z,\mathrm{Fix}(T))<\mu.
\end{equation}
By (\ref{Sa (1.5)}) and since $T$ is nonexpansive, for all sufficiently large
integers $k$,%
\begin{equation}
d(T^{j}x^{k},\mathrm{Fix}(T))<\mu,
\end{equation}
contradicting (\ref{Sa (1.4)}), thus, concluding the proof.
\end{proof}

\begin{theorem}
\label{Sa Theo 1.4}Let $X\ $be a complete and compact metric space, assume
that $T:X\rightarrow X$ is a nonexpansive operator for which Condition
\ref{cond:condition} holds, and let $\varepsilon>0$. Then there exists a
natural number $k_{0}$ such that for each $x\in X$ and each integer $k\geq
k_{0}$,%
\begin{equation}
\rho(T^{k}x,\lim_{i\rightarrow\infty}T^{i}x)\leq\varepsilon.
\end{equation}

\end{theorem}

\begin{proof}
By Lemma \ref{Sa Lem 1.5}, there exists an integer $k_{0}$ such that for each
$x\in X$ there exists $j\in\{0,1,\dots,k_{0}\}$ so that%
\begin{equation}
d(T^{j}x,\mathrm{Fix}(T))<\varepsilon/2. \label{Sa Property (a)}%
\end{equation}
This implies that there exist a%
\begin{equation}
j\in\{0,1,\dots,k_{0}\} \label{Sa (1.6)}%
\end{equation}
and a%
\begin{equation}
z\in\mathrm{Fix}(T) \label{Sa (1.7)}%
\end{equation}
such that
\begin{equation}
\rho(T^{j}x,z)<\varepsilon/2. \label{Sa (1.8)}%
\end{equation}
Since $T$ is nonexpansive, we get, by (\ref{Sa (1.7)}) and (\ref{Sa (1.8)}),
that for all integers $k\geq j$,%
\begin{equation}
\rho(T^{k}x,z)\leq\rho(T^{j}x,z)<\varepsilon/2, \label{Sa (1.9)}%
\end{equation}
yielding,%
\begin{equation}
\rho(\lim_{i\rightarrow\infty}T^{i}x,z)\leq\varepsilon/2.
\end{equation}
Together with (\ref{Sa (1.6)}) and (\ref{Sa (1.9)}) this implies that for all
integers $k\geq k_{0}$,%
\begin{equation}
\rho(T^{k}x,\lim_{i\rightarrow\infty}T^{i}x)\leq\rho(T^{k}x,z)+\rho
(z,\lim_{i\rightarrow\infty}T^{i}x)\leq\varepsilon/2+\varepsilon
/2=\varepsilon,
\end{equation}
completing the proof.
\end{proof}

Theorem \ref{Sa Theo 1.4} implies the following theorem:

\begin{theorem}
\label{Sa Theorem 1.6 - 20200426}Let $X\ $be a complete metric space, assume
that $T:X\rightarrow X$ is a nonexpansive operator for which Condition
\ref{cond:condition} holds, the closure of $T\left(  X\right)  $ is compact,
and let $\varepsilon>0$. Then there exists a natural number $k_{0}$ such that
for each $x\in X$ and each integer $k\geq k_{0}$,%
\begin{equation}
\rho(T^{k}x,\lim_{i\rightarrow\infty}T^{i}x)\leq\varepsilon.
\end{equation}

\end{theorem}

\begin{proposition}
\label{Sa Prop 1.7}Under the assumptions of Theorem
\ref{Sa Theorem 1.6 - 20200426}, there exist an integer $k_{0}$ and a
$\delta>0$ such that for each finite sequence $\{x^{i}\}_{i=0}^{k_{0}}\subset
X$ satisfying%
\begin{equation}
\rho(x^{i+1},Tx^{i})\leq\delta,\text{ for all }i=0,1,\dots,k_{0}-1,
\label{eq:cond-prop-1.7}%
\end{equation}
the inequality%
\begin{equation}
d(x^{k_{0}},\mathrm{Fix}(T))\leq\varepsilon
\end{equation}
holds.
\end{proposition}

\begin{proof}
Theorem \ref{Sa Theorem 1.6 - 20200426} implies that there exists an integer
$k_{0}$ such that for each $x\in X$,%
\begin{equation}
d(T^{k_{0}}x,\mathrm{Fix}(T))\leq\varepsilon/4. \label{Sa (1.10)}%
\end{equation}
Define%
\begin{equation}
\delta:=4^{-1}\varepsilon(k_{0})^{-1}, \label{Sa (1.11)}%
\end{equation}
assume that $\{x^{i}\}_{i=0}^{k_{0}}\subset X$ satisfies
(\ref{eq:cond-prop-1.7}) and set%
\begin{equation}
y^{0}:=x^{0},\;y^{i+1}:=Ty^{i},\;\text{for all }i=0,1,\dots,k_{0}-1.
\label{Sa (1.13)}%
\end{equation}
In view of (\ref{Sa (1.10)}) and (\ref{Sa (1.13)}),%
\begin{equation}
d(y^{k_{0}},\mathrm{Fix}(T))\leq\varepsilon/4. \label{Sa (1.14)}%
\end{equation}
Next we show, by induction, that%
\begin{equation}
\rho(y^{i},x^{i})\leq i\delta,\;\text{for all }i=0,1,\dots,k_{0}.
\label{Sa (1.15)}%
\end{equation}
Equation (\ref{Sa (1.13)}) implies that (\ref{Sa (1.15)}) holds for $i=0$. Let
$i\in\{0,1,\dots,k_{0}-1\}$ for which (\ref{Sa (1.15)}) holds. By the
nonexpansiveness of $T$, (\ref{eq:cond-prop-1.7}), (\ref{Sa (1.13)}) and
(\ref{Sa (1.15)}),%
\begin{align}
\rho(y^{i+1},x^{i+1})  &  =\rho(Ty^{i},x^{i+1})\nonumber\\
&  \leq\rho(Ty^{i},Tx^{i})+\rho(Tx^{i},x^{i+1})\leq\rho(y^{i},x^{i}%
)+\delta\leq(i+1)\delta,
\end{align}
in particular,%
\begin{equation}
\rho(y^{k_{0}},x^{k_{0}})\leq k_{0}\delta. \label{Sa (1.16)}%
\end{equation}
It follows now from (\ref{Sa (1.11)}), (\ref{Sa (1.14)}) and (\ref{Sa (1.16)})
that%
\begin{equation}
d(x^{k_{0}},\mathrm{Fix}(T))\leq d(x^{k_{0}},y^{k_{0}})+d(y^{k_{0}%
},\mathrm{Fix}(T))\leq\varepsilon/4+\varepsilon/4,
\end{equation}
which concludes the proof.
\end{proof}

\begin{theorem}
\label{Sa Theo 1.6}Under the assumptions of Theorem
\ref{Sa Theorem 1.6 - 20200426}, there exist an integer $k_{0}$ and a
$\delta>0$ such that for each sequence $\{x^{i}\}_{i=0}^{\infty}\subset X$
satisfying $\rho(x^{i+1},Tx^{i})\leq\delta$, for all $i=0,1,\dots,$ the
inequality%
\begin{equation}
d(x^{i},\mathrm{Fix}(T))\leq\varepsilon
\end{equation}
holds for all integers $i\geq k_{0}.$
\end{theorem}

\begin{proof}
The proof follows from Proposition \ref{Sa Prop 1.7}.
\end{proof}

Theorem \ref{Sa Theo 1.6} implies the next result.

\begin{theorem}
\label{thm:thm5.1}Under the assumptions of Theorem
\ref{Sa Theorem 1.6 - 20200426}, if we take a sequence%
\begin{equation}
\{\mu_{k}\}_{k=1}^{\infty}\subset(0,\infty),\text{ }\lim_{k\rightarrow\infty
}\mu_{k}=0, \label{eq:5.3}%
\end{equation}
then there exists an integer $k_{1}>0$ such that for each sequence
$\{x^{i}\}_{i=0}^{\infty}\subset X$ satisfying
\begin{equation}
\rho(x^{i+1},Tx^{i})\leq\mu_{i+1},\;i=0,1,\dots, \label{eq:5.3 second}%
\end{equation}
the inequality $d(x^{k},\mathrm{Fix}(T))\leq\varepsilon$ holds for all
integers $k\geq k_{1}$.
\end{theorem}

\section{Data-compatibility of the hybrid subgradient method
\label{sec:convergence}}

Our main data-compatibility result in Theorem \ref{thm:7.1} below is obtained
under the following assumptions: $T:H\rightarrow H$ is a nonexpansive operator
for which Condition \ref{cond:condition} holds, the closure of $T\left(
H\right)  $ (which denoted by $cl\left(  T\left(  H\right)  \right)  $) is
compact, $f:H\rightarrow R$ convex function, Lipschitz on any bounded set.

Since $cl\left(  T\left(  H\right)  \right)  $ is compact, there exists a ball
$B(0,M),$ with $M>0,$ such that%
\begin{equation}
\mathrm{Fix}\left(  T\right)  \subset cl\left(  T\left(  H\right)  \right)
\subset B(0,M), \label{eq 52 N}%
\end{equation}
which means that the set $\mathrm{Fix}(T)$ is bounded.

Proposition \ref{The limit is fixed point} yields that $\mathrm{Fix}(T)$ is
closed. Since $\mathrm{Fix}(T)$ is closed subset of a compact set, which is
$cl\left(  T\left(  H\right)  \right)  $, then it easy to see that
$\mathrm{Fix}(T)$ is compact also. By Condition \ref{cond:condition} and
Proposition \ref{The limit is fixed point} we get that $\mathrm{Fix}(T)$ is
nonempty. The convexity of $f$ on $H$ implies the continuity of $f$ on the
space\ $H$ and especially on the subset $\mathrm{Fix}(T)$. The continuity of
$f$ on $\mathrm{Fix}(T)$ and the compactness of $\mathrm{Fix}(T),$ implies
that there exists some point $x\in\operatorname*{SOL}(f,\mathrm{Fix}(T)),$
i.e., $\operatorname*{SOL}(f,\mathrm{Fix}(T))\neq\varnothing$.

Since $f$ is Lipschitz on any bounded set, there exists a number $\bar{L}>1$
such that%
\begin{equation}
|f(z^{1})-f(z^{2})|\leq\bar{L}||z^{1}-z^{2}||,{\text{ for all }}z^{1},z^{2}\in
B(0,3M+2). \label{eq 53 N}%
\end{equation}

We need the following lemma to prove the main result.

\begin{lemma}
{\label{lem-8.3}}Assume $T:H\rightarrow H$ is a nonexpansive operator for
which Condition \ref{cond:condition} holds and $cl\left(  T\left(  H\right)
\right)  $ is compact. Let $f:H\rightarrow R$ convex function and Lipschitz on
any bounded set. Let $\bar{x}\in\operatorname*{SOL}(f,\mathrm{Fix}(T))$ and
let $\Delta\in(0,1],$ $\alpha>0$ and $x\in cl\left(  T\left(  H\right)
\right)  $ satisfy%
\begin{equation}
\left\Vert x\right\Vert \leq3M+2,\;f(x)>f(\bar{x})+\Delta, \label{eq:8.3}%
\end{equation}
where $M$ is as in (\ref{eq 52 N}). Further, let $v\in\partial f(x).$ Then
$v\not =0$ and%
\begin{equation}
y:=T\left(  x-\alpha||v||^{-1}v\right)
\end{equation}
satisfies%
\begin{equation}
\Vert y-\bar{x}\Vert^{2}\leq\Vert x-\bar{x}\Vert^{2}-2\alpha(4\bar{L}%
)^{-1}\Delta+\alpha^{2},
\end{equation}
where $\bar{L}$ is as in (\ref{eq 53 N}). Moreover,%
\begin{equation}
d(y,\operatorname*{SOL}(f,\mathrm{Fix}(T)))^{2}\leq d(x,\operatorname*{SOL}%
(f,\mathrm{Fix}(T)))^{2}-2\alpha(4\bar{L})^{-1}\Delta+\alpha^{2}.
\end{equation}

\end{lemma}

\begin{proof}
From (\ref{eq:8.3}) $v\not =0$. For $\bar{x}\in\operatorname*{SOL}%
(f,\mathrm{Fix}(T))$, we have, by (\ref{eq 53 N}) and (\ref{eq 52 N}), that
for each $z\in B(\bar{x},4^{-1}\Delta\bar{L}^{-1})$,%
\begin{equation}
f(z)\leq f(\bar{x})+\bar{L}||z-\bar{x}||\leq f(\bar{x})+4^{-1}\Delta.
\end{equation}
Therefore, (\ref{eq:8.3}) and $v\in\partial f(x)$, imply that%
\begin{equation}
\left\langle v,z-x\right\rangle \leq f(z)-f(x)\leq-(3/4)\Delta,\text{ for all
}z\in B(\bar{x},4^{-1}\Delta\bar{L}^{-1}).
\end{equation}
From this inequality we deduce that%
\begin{equation}
\left\langle ||v||^{-1}v,z-x\right\rangle <0,{\text{ for all }}z\in B(\bar
{x},4^{-1}\Delta\bar{L}^{-1}),
\end{equation}
or, setting $\bar{z}:=\bar{x}+4^{-1}\bar{L}^{-1}\Delta\Vert v\Vert^{-1}v,$
that%
\begin{equation}
0>\left\langle ||v||^{-1}v,\bar{z}-x\right\rangle =\left\langle ||v||^{-1}%
v,\bar{x}+4^{-1}\bar{L}^{-1}\Delta\Vert v\Vert^{-1}v-x\right\rangle .
\end{equation}
This leads to%
\begin{equation}
\left\langle ||v||^{-1}v,\bar{x}-x\right\rangle <-4^{-1}\bar{L}^{-1}\Delta.
\end{equation}
Putting $\tilde{y}:=x-\alpha\Vert v\Vert^{-1}v,$ we arrive at%
\begin{align}
\Vert\tilde{y}-\bar{x}\Vert^{2}  &  =\Vert x-\alpha\Vert v\Vert^{-1}v-\bar
{x}\Vert^{2}\nonumber\\
&  =\Vert x-\bar{x}\Vert^{2}-2\left\langle x-\bar{x},\alpha\Vert v\Vert
^{-1}v\right\rangle +\alpha^{2}\nonumber\\
&  \leq\Vert x-\bar{x}\Vert^{2}-2\alpha(4\bar{L})^{-1}\Delta+\alpha^{2}.
\end{align}
From all the above we obtain%
\begin{align}
\Vert y-\bar{x}\Vert^{2}  &  =\Vert T\tilde{y}-\bar{x}\Vert^{2}\leq\Vert
\tilde{y}-\bar{x}\Vert^{2}\nonumber\\
&  \leq\Vert x-\bar{x}\Vert^{2}-2\alpha(4\bar{L})^{-1}\Delta+\alpha^{2},
\end{align}
which completes the proof.
\end{proof}

Now we present the main theorem showing that sequences generated by the hybrid
subgradient method (HSM) have a $(\tau,\bar{L})$-output, i.e., contain an
iterate that is data-compatible. The condition (\ref{eq:7.3}) in this and
subsequent theorems is common in many results in optimization theory, see,
e.g., Theorem 2.2 and many subsequent theorems in N. Shor's classic book
\cite{Shor1985}.

\begin{theorem}
{\label{thm:7.1}} Assume $T:H\rightarrow H$ is a nonexpansive operator for
which Condition \ref{cond:condition} holds and $cl\left(  T\left(  H\right)
\right)  $ is compact. Let $f:H\rightarrow R$ convex function and Lipschitz on
any bounded set. Let%
\begin{equation}
\{\alpha_{k}\}_{k=0}^{\infty}\subset(0,1],\text{ be a sequence such that }%
\lim_{k\rightarrow\infty}\alpha_{k}=0\text{ and}\;\sum_{k=0}^{\infty}%
\alpha_{k}=\infty, \label{eq:7.3}%
\end{equation}
let $\bar{L}$ be fixed, as defined by (\ref{eq 53 N}), and let $\tau\in(0,1)$.
Then there exists an integer $K$ such that for any sequence $\{x^{k}%
\}_{k=1}^{\infty}\subset cl\left(  T\left(  H\right)  \right)  $, generated by
Algorithm \ref{alg:sa-psm} , the inequalities%
\begin{gather}
d(x^{k},\operatorname*{SOL}(f,\mathrm{Fix}\left(  T\right)  ))\leq\tau{\text{
}}\\
{\text{and }}\nonumber\\
f(x^{k})\leq f(z)+\tau\bar{L}\text{ for all }z\in\operatorname*{SOL}%
(f,\mathrm{Fix}\left(  T\right)  )
\end{gather}
hold for all integers $k\geq K$.
\end{theorem}

\begin{proof}
Fix an $\bar{x}\in\operatorname*{SOL}(f,\mathrm{Fix}(T)).$ It is not difficult
to see that there exists a number $\tau_{0}\in(0,\tau/4)$ such that for each
$x\in cl\left(  T\left(  H\right)  \right)  $ satisfying $d(x,\mathrm{Fix}%
(T))\leq\tau_{0}$ and $f(x)\leq f(\bar{x})+\tau_{0}$ we have%
\begin{equation}
d(x,\operatorname*{SOL}(f,\mathrm{Fix}(T)))\leq\tau/4. \label{eq:P2}%
\end{equation}
Since $\{x^{k}\}_{k=1}^{\infty}$ is generated by Algorithm \ref{alg:sa-psm} we
know, from (\ref{eq:alg-sa-psm-1}), (\ref{eq:alg-sa-psm-2}) and (\ref{eq:1.8}%
), that%
\begin{equation}
\left\Vert x^{k}-Tx^{k-1}\right\Vert \leq\alpha_{k-1},\text{ for all }k\geq1.
\label{eq:8.23}%
\end{equation}
Thus, by Theorem \ref{thm:thm5.1} and (\ref{eq:7.3}), there exists an integer
$n_{1}$ such that%
\begin{equation}
d(x^{k},\mathrm{Fix}(T))\leq\tau_{0},\text{ for all }k\geq n_{1}.
\label{eq:8.24}%
\end{equation}
This, along with (\ref{eq 52 N}), guarantees that%
\begin{equation}
\Vert x^{k}\Vert\leq M+1,{\text{ for all }}k\geq n_{1}. \label{eq:8.25}%
\end{equation}
Choose a positive $\tau_{1}$ for which $\tau_{1}<(8\bar{L})^{-1}\tau_{0},$ by
(\ref{eq:7.3}) there is an integer $n_{2}>n_{1}$ such that%
\begin{equation}
\alpha_{k}\leq\tau_{1}(32)^{-1},{\text{ for all }}k>n_{2}, \label{eq:8.16}%
\end{equation}
and so, there is an integer $n_{0}>n_{2}+4$ such that%
\begin{equation}
\sum_{k=n_{2}}^{n_{0}-1}\alpha_{k}>8(2M+1)^{2}\bar{L}\tau_{0}^{-1}.
\label{eq:8.17}%
\end{equation}

We show now that there exists an integer $p\in\lbrack n_{2}+1,n_{0}]$ such
that $f(x^{p})\leq f(\bar{x})+\tau_{0}$. Assuming the contrary means that for
all $k\in\lbrack n_{2}+1,n_{0}]$,
\begin{equation}
f(x^{k})>f(\bar{x})+\tau_{0}. \label{eq:8.26}%
\end{equation}
By (\ref{eq:8.26}), (\ref{eq:7.3}), (\ref{eq:8.25}) and using Lemma
\ref{lem-8.3}, with $\Delta=\tau_{0}$, $\alpha=\alpha_{k}$, $x=x^{k}$,
$y=x^{k+1}$, $v=s^{k}$, we get, for all $k\in\lbrack n_{2}+1,n_{0}]$,%
\begin{align}
&  d(x^{k+1},\operatorname*{SOL}(f,\mathrm{Fix}(T)))^{2}\nonumber\\
&  \leq d(x^{k},\operatorname*{SOL}(f,\mathrm{Fix}(T)))^{2}-2\alpha_{k}%
(4\bar{L})^{-1}\tau_{0}+\alpha_{k}^{2}.
\end{align}
According to the choice of $\tau_{1}$ and by (\ref{eq:8.16}) this implies that
for all $k\in\lbrack n_{2}+1,n_{0}]$,%
\begin{align}
&  d(x^{k},\operatorname*{SOL}(f,\mathrm{Fix}(T)))^{2}-d(x^{k+1}%
,\operatorname*{SOL}(f,\mathrm{Fix}(T)))^{2}\nonumber\\
&  \geq\alpha_{k}[(2\bar{L})^{-1}\tau_{0}-\alpha_{k}]\nonumber\\
&  \geq\alpha_{k}(4\bar{L})^{-1}\tau_{0},
\end{align}
which, together with (\ref{eq:8.25}) and (\ref{eq 52 N}), gives%
\begin{align}
&  (2M+1)^{2}\nonumber\\
&  \geq d(x^{n_{2}+1},\operatorname*{SOL}(f,\mathrm{Fix}(T)))^{2}\nonumber\\
&  \geq\sum_{k=n_{2}+1}^{n_{0}}\left(  d(x^{k},\operatorname*{SOL}%
(f,\mathrm{Fix}(T)))^{2}-d(x^{k+1},\operatorname*{SOL}(f,\mathrm{Fix}%
(T)))^{2}\right) \nonumber\\
&  \geq(4\bar{L})^{-1}\tau_{0}\sum_{k=n_{2}+1}^{n_{0}}\alpha_{k},
\end{align}
and%
\begin{equation}
\sum_{k=n_{2}+1}^{n_{0}}\alpha_{k}\leq(2M+1)^{2}4\bar{L}\tau_{0}^{-1}.
\end{equation}
This contradicts (\ref{eq:8.17}), proving that there is an integer
$p\in\lbrack n_{2}+1,n_{0}]$ such that $f(x^{p})\leq f(\bar{x})+\tau_{0}$.
Thus, by (\ref{eq:8.24}) and (\ref{eq:P2}),%
\begin{equation}
d(x^{p},\operatorname*{SOL}(f,\mathrm{Fix}(T)))\leq\tau/4.
\end{equation}
We show that for all $k\geq p$, $d(x^{k},\operatorname*{SOL}(f,\mathrm{Fix}%
(T)))\leq\tau$. Assuming the contrary,%
\begin{equation}
\text{there exists a }q>p\text{ such that }d(x^{q},\operatorname*{SOL}%
(f,\mathrm{Fix}(T)))>\tau. \label{eq:8.31}%
\end{equation}
We may assume, without loss of generality, that%
\begin{equation}
d(x^{k},\operatorname*{SOL}(f,\mathrm{Fix}(T)))\leq\tau,{\text{ for all }%
}p\leq k<q. \label{eq:8.32}%
\end{equation}
One of the following two cases must hold: (i) $f(x^{q-1})\leq f(\bar{x}%
)+\tau_{0},$ or (ii) $f(x^{q-1})>f(\bar{x})+\tau_{0}.$ In case (i), since
$p\in\lbrack n_{2}+1,n_{0}],$ (\ref{eq:8.24}), (\ref{eq:8.25}) and
(\ref{eq:P2}) show that%
\begin{equation}
d(x^{q-1},\operatorname*{SOL}(f,\mathrm{Fix}(T)))\leq\tau/4.
\end{equation}
Thus, there is a point $z\in\operatorname*{SOL}(f,\mathrm{Fix}(T))$ such that
$\left\Vert x^{q-1}-z\right\Vert <\tau/3.$ Using this fact and (\ref{eq:8.23}%
), (\ref{eq:1.8}), (\ref{eq:1.9}) and (\ref{eq:8.16}), yields%
\begin{align}
\left\Vert x^{q}-z\right\Vert  &  \leq\left\Vert x^{q}-Tx^{q-1}\right\Vert
+\left\Vert Tx^{q-1}-z\right\Vert \nonumber\\
&  \leq\alpha_{q-1}+\left\Vert x^{q-1}-z\right\Vert \leq\tau/4+\tau/3,
\end{align}
proving that $d(x^{q},\operatorname*{SOL}(f,\mathrm{Fix}(T)))\leq\tau.$ This
contradicts (\ref{eq:8.31}) and implies that case (ii) must hold, namely that
$f(x^{q-1})>f(\bar{x})+\tau_{0}$. This, along with (\ref{eq:8.25}),
(\ref{eq:8.16}), the choice of $\tau_{1}$, (\ref{eq:8.32}) and Lemma
\ref{lem-8.3}, with $\Delta=\tau_{0}$, $\alpha=\alpha_{q-1}$, $x=x^{q-1}$,
$y=x^{q}$, shows that%
\begin{align}
&  d(x^{q},\operatorname*{SOL}(f,\mathrm{Fix}(T)))^{2}\nonumber\\
&  \leq d(x^{q-1},\operatorname*{SOL}(f,\mathrm{Fix}(T)))^{2}-2\alpha
_{q-1}(4\bar{L})^{-1}\tau_{0}+\alpha_{q-1}^{2}\nonumber\\
&  \leq d(x^{q-1},\operatorname*{SOL}(f,\mathrm{Fix}(T)))^{2}-\alpha
_{q-1}((2\bar{L})^{-1}\tau_{0}-\alpha_{q-1})\nonumber\\
&  \leq d(x^{q-1},\operatorname*{SOL}(f,\mathrm{Fix}(T)))^{2}\leq\tau^{2},
\end{align}
namely, that $d(x^{q},\operatorname*{SOL}(f,\mathrm{Fix}(T)))\leq\tau.$ This
contradicts (\ref{eq:8.31}), proving that, for all $k\geq p$, $d(x^{k}%
,\operatorname*{SOL}(f,\mathrm{Fix}(T)))\leq\tau$. Together with
(\ref{eq 52 N}) and (\ref{eq 53 N}) this implies that, for all $k\geq n_{0}$,%
\begin{equation}
f(x^{k})\leq f(z)+\tau\bar{L}\text{ for all }z\in\operatorname*{SOL}%
(f,\mathrm{Fix}\left(  T\right)  ),
\end{equation}
and the proof is complete.
\end{proof}

\section{\textbf{The string-averaging hybrid subgradient }\newline%
\textbf{method }(SA-HSM)\textbf{ }\label{sect:SAPv}}

Assume that $O_{1},O_{2},\dots,O_{m}$ are nonexpansive operators mapping $H$
into $H,$ for which%
\begin{equation}
\mathcal{F}:=%
{\textstyle\bigcap\limits_{i=1}^{m}}
\mathrm{Fix}\left(  O_{i}\right)  \neq\varnothing\label{Non Empty Inter}%
\end{equation}
Let $f:H\rightarrow R$ be a convex function and Lipschitz on any bounded set.
We are interested in solving the following problem by using a string-averaging
algorithmic scheme.%
\begin{equation}
\min\{f(x)\mid x\in\mathcal{F}\} \label{Prov_inter 1}%
\end{equation}
whose solution means to%
\begin{equation}
\text{find a point }x\text{ in }\operatorname*{SOL}(f,\mathcal{F}),
\label{Prov_inter  2}%
\end{equation}
where%

\begin{equation}
\operatorname*{SOL}(f,\mathcal{F}):=\{x\in\mathcal{F}\mid f(x)\leq
f(y),{\text{ for all }}y\in\mathcal{F}\}. \label{Sol_Def_Inter}%
\end{equation}

For $t=1,2,\dots,\Theta,$ let the \textit{string} $I_{t}$ be an ordered subset
of $\{1,2,\dots,m\}$ of the form%
\begin{equation}
I_{t}=(i_{1}^{t},i_{2}^{t},\dots,i_{m(t)}^{t}), \label{block}%
\end{equation}
with $m(t)$ the number of elements in $I_{t}.$ For any $x\in H,$ the product
of operators along a string $I_{t},$ $t=1,2,\dots,\Theta,$ is
\begin{equation}
F_{t}(x):=O_{i_{m(t)}^{t}}\cdots O_{i_{2}^{t}}O_{i_{1}^{t}}(x),
\label{notation 1}%
\end{equation}
and is called a \textquotedblleft string operator\textquotedblright.

We deal with string-averaging of fixed strings and fixed weights. To this end
we assume that%
\begin{equation}
\{1,2,\dots,m\}\subset%
{\displaystyle\bigcup\limits_{t=1}^{\Theta}}
I_{t} \label{eq:contain}%
\end{equation}
and that a system of nonnegative weights $w_{1,}w_{2},\cdots,w_{\Theta}$ such
that $\sum_{t=1}^{\Theta}w_{t}=1$ is fixed and given. We define the operator%
\begin{equation}
O(x):=\sum_{t=1}^{\Theta}w_{t}F_{t}(x). \label{eq:sum}%
\end{equation}
This operator will be called \textquotedblleft fit\textquotedblright\ if the
strings that define it obey (\ref{eq:contain}). We will need the following condition.

\begin{condition}
\label{Strict inequality}For all $i=1,2,\ldots,m,$ the following holds: For
any $y\in H\backslash\mathrm{Fix}\left(  O_{i}\right)  $ there exist
$x\in\mathcal{F=}%
{\textstyle\bigcap\limits_{i=1}^{m}}
\mathrm{Fix}\left(  O_{i}\right)  $ such that $\left\Vert O_{i}\left(
y\right)  -x\right\Vert <\left\Vert y-x\right\Vert .$
\end{condition}

\begin{proposition}
\label{FixO_Eq_FixInt}Let $O_{1},O_{2},\dots,O_{m}$ be nonexpansive operators
$O_{i}:H\rightarrow H$, and let $O=\sum_{t=1}^{\Theta}w_{t}F_{t}(x),$ be as in
(\ref{eq:sum}). If (\ref{eq:contain}) and condition \ref{Strict inequality}
hold, then $\mathrm{Fix}\left(  O\right)  =\mathcal{F}.$
\end{proposition}

\begin{proof}
Clearly, $\mathcal{F}\subset\mathrm{Fix}\left(  O\right)  ,$ therefore, it is
sufficient to prove that $\mathrm{Fix}\left(  O\right)  \subset\mathcal{F}.$
Assume by negation that $\widehat{y}\in\mathrm{Fix}x\left(  O\right)  $ such
that $\widehat{y}\notin\mathcal{F}.$ This means that there is an
$1\leq\widehat{i}\leq m$ such that $\widehat{y}\notin\mathrm{Fix}\left(
O_{\widehat{i}}\right)  .$ Condition \ref{Strict inequality} implies that
there exist an $\overline{x}\in\mathcal{F}$ that satisfies $\left\Vert
O_{\widehat{i}}\left(  \widehat{y}\right)  -\overline{x}\right\Vert
<\left\Vert \widehat{y}-\overline{x}\right\Vert $. From this inequality, since
$O_{1},O_{2},\dots,O_{m}$ are nonexpansive operators, it is easy to see that%
\begin{equation}
\left\Vert O\left(  \widehat{y}\right)  -\overline{x}\right\Vert =\left\Vert
\sum_{t=1}^{\Theta}w_{t}F_{t}\left(  \widehat{y}\right)  -\overline
{x}\right\Vert \leq\sum_{t=1}^{\Theta}w_{t}\left\Vert F_{t}\left(  \widehat
{y}\right)  -\overline{x}\right\Vert <\left\Vert \widehat{y}-\overline
{x}\right\Vert ,
\end{equation}
and, consequently, that $\widehat{y}\notin\mathrm{Fix}\left(  O\right)  .$
This contradicts the negation assumption made above and completes the proof.
\end{proof}

We propose the following string-averaging hybrid subgradient method (SA-HSM)
for solving the problem (\ref{Prov_inter 1}).

\begin{algorithm}
\label{String-alg:sa-psm}\textbf{String-Averaging Hybrid Subgradient Method
(SA-HSM).}

\textbf{Initialization}: Let $\{\alpha_{k}\}_{k=0}^{\infty}\subset(0,1]$ be a
scalar sequence and let $x^{0}\in H$ be an arbitrary initialization vector.

\textbf{Iterative step}: Given a current iteration vector $x^{k}$ calculate
the next vector as follows:

Choose any $s^{k}\in\partial f(x^{k})$ and calculate%
\begin{equation}
x^{k+1}=O\left(  x^{k}-\alpha_{k}\frac{s^{k}}{\parallel s^{k}\parallel
}\right)  \text{,} \label{eq:alg-sa-psm-2_string}%
\end{equation}

but if $s^{k}=0$ then set $\frac{\textstyle s^{k}}{\textstyle\parallel
s^{k}\parallel}:=0.$
\end{algorithm}

The next theorem shows that sequences generated by the string-averaging hybrid
subgradient method (SA-HSM) have a $(\tau,\bar{L})$-output, i.e., contain an
iterate that is data-compatible.

\begin{theorem}
{\label{thm:7.1 string}}Let $O_{1},O_{2},\dots,O_{m}$ be nonexpansive
operators mapping $H$ into $H,$ such that $\mathcal{F}=%
{\textstyle\bigcap\limits_{i=1}^{m}}
\mathrm{Fix}\left(  O_{i}\right)  \neq\varnothing$ and $cl\left(  O_{i}\left(
H\right)  \right)  $ is compact for all $i=1,2,\ldots,m.$ Let $O=\sum
_{t=1}^{\Theta}w_{t}F_{t}(x)$ be as in (\ref{eq:sum}) and assume that
$\lim_{j\rightarrow\infty}O^{j}y^{0}$ exists for any $y^{0}\in H$. Let
$f:H\rightarrow R$ be a convex function and Lipschitz on any bounded set. Let
$\{\alpha_{k}\}_{k=0}^{\infty}\subset(0,1]$ be a sequence such that%
\begin{equation}
\lim_{k\rightarrow\infty}\alpha_{k}=0\text{ and}\;\sum_{k=0}^{\infty}%
\alpha_{k}=\infty, \label{eq:7.3_String}%
\end{equation}
let $\bar{L}$ be fixed, as defined by (\ref{eq 53 N}), and let $\tau\in(0,1)$.
If (\ref{eq:contain}) and condition \ref{Strict inequality} hold then there
exists an integer $K$ such that for any sequence $\{x^{k}\}_{k=0}^{\infty
}\subset H$, generated by Algorithm \ref{String-alg:sa-psm}, the inequalities%
\begin{gather}
d(x^{k},\operatorname*{SOL}(f,\mathcal{F})\leq\tau{\text{ }}\\
{\text{and }}\nonumber\\
f(x^{k})\leq f(z)+\tau\bar{L}\text{ for all }z\in\operatorname*{SOL}%
(f,\mathcal{F})
\end{gather}
hold for all integers $k\geq K$.
\end{theorem}

\begin{proof}
Since $O_{1},O_{2},\dots,O_{m}$ are nonexpansive and $O$ is a fit operator, it
follows that that $O$ is nonexpansive. Moreover, condition
\ref{Strict inequality} and proposition \ref{FixO_Eq_FixInt} ensure that
$\mathrm{Fix}\left(  O\right)  =\mathcal{F}.$ This, along with the other
assumptions of the theorem, enable the use of Theorem \ref{thm:7.1} to
complete the proof.
\end{proof}

\section{Data-compatibility for constrained minimization with inconsistent
constraints\label{sect:inconsistent}}

In this section we consider a data pair $(\Gamma,f),$ assuming that
$\Gamma:=\{C_{i}\}_{i=1}^{m}$ is a family of closed and convex subsets of $H,$
not necessarily obeying that $C:=\cap_{i=1}^{m}C_{i}\neq\varnothing$. Let
$\{w_{i}\}_{i=1}^{m}$ be a set of weights such that $w_{i}\geq0$ and
$\sum_{i=1}^{m}w_{i}=1.$ It is well-known that the operator $P_{w}:=\sum
_{i=1}^{m}w_{i}P_{C_{i}}$ is nonexpansive and satisfies%
\begin{equation}
\mathrm{Fix}\left(  P_{w}\right)  =\operatorname*{Arg}\min\{\text{Prox}%
_{\Gamma}(x)\mid x\in H\},
\end{equation}
where Prox$_{\Gamma}(x):=\frac{1}{2}\sum_{i=1}^{m}w_{i}\left\Vert P_{C_{i}%
}(x)-x\right\Vert ^{2},$ see the succinct \cite[Subsection 5.4]%
{Cegielski2012Book} on the simultaneous projection method. If $C\neq
\varnothing$ then $\mathrm{Fix}\left(  P_{w}\right)  =C.$ If, however,
$C=\varnothing$ then $\mathrm{Fix}\left(  P_{w}\right)  =\Pi(\Gamma,\gamma)$
for $\gamma=\min\{$Prox$_{\Gamma}(x)\mid x\in H\}$ and is nonempty. Moreover,
for any $y^{0}\in H$ the limit $\lim_{k\rightarrow\infty}(P_{w})^{k}y^{0}$
exists and belong to $\mathrm{Fix}\left(  P_{w}\right)  $.

Consider the following algorithm.

\begin{algorithm}
\label{alg:Sim Proj sa-psm}\textbf{Simultaneous Projection Hybrid Subgradient
Method (SP-HSM).}

\textbf{Initialization}: Let $\{\alpha_{k}\}_{k=0}^{\infty}\subset(0,1]$ be a
scalar sequence and let $x^{0}\in H$ be an arbitrary initialization vector.

\textbf{Iterative step}: Given a current iteration vector $x^{k}$ calculate
the next vector as follows:

Choose any $s^{k}\in\partial f(x^{k})$ and calculate%
\begin{equation}
x^{k+1}=P_{w}\left(  x^{k}-\alpha_{k}\frac{s^{k}}{\parallel s^{k}\parallel
}\right)  \text{,}%
\end{equation}

but if $s^{k}=0$ then set $\frac{\textstyle s^{k}}{\textstyle\parallel
s^{k}\parallel}:=0.$
\end{algorithm}

From the above assumptions and discussion we obtain the following theorem as a
consequence of Theorem \ref{thm:7.1}. It does not assume consistency of the
underlying constraints $\Gamma=\{C_{i}\}_{i=1}^{m},$ and shows that sequences
generated by the simultaneous projection hybrid subgradient method (SP-HSM)
have a $(\tau,\bar{L})$-output, i.e., contain an iterate that is
data-compatible. Observe in the next theorem that $\overline{L}$ is a fixed
constant that obeys (\ref{eq 53 N}) and that the parameter $\tau$ must obey
$\tau\in(0,1),$ so, once $\overline{L}$ is fixed the \textquotedblleft
user\textquotedblright\ can choose a small $\tau$ so that $\tau\overline{L}$
is as small as he wants.

\begin{theorem}
\label{thm:7.1 Proj}Assume that $cl\left(  P_{w}\left(  H\right)  \right)  $
is compact. Let $f:H\rightarrow R$ be a convex function and Lipschitz on any
bounded set, let%
\begin{equation}
\{\alpha_{k}\}_{k=0}^{\infty}\subset(0,1],\text{ be a sequence such that }%
\lim_{k\rightarrow\infty}\alpha_{k}=0\text{ and}\;\sum_{k=0}^{\infty}%
\alpha_{k}=\infty,
\end{equation}
let $\bar{L}$ be fixed, as defined by (\ref{eq 53 N}), and let $\tau\in(0,1)$.
Then there exists an integer $K$ such that, for any sequence $\{x^{k}%
\}_{k=0}^{\infty}\subset H$ generated by Algorithm \ref{alg:Sim Proj sa-psm},
the inequalities%
\begin{gather}
d(x^{k},\operatorname*{SOL}(f,\mathrm{Fix}\left(  P_{w}\right)  ))\leq
\tau{\text{ }}\\
{\text{and }}\nonumber\\
f(x^{k})\leq f(z)+\tau\bar{L}\text{ for all }z\in\operatorname*{SOL}%
(f,\mathrm{Fix}\left(  P_{w}\right)  )
\end{gather}
hold for all integers $k\geq K$.\bigskip
\end{theorem}

\section{Minimization over disjoint hard and soft constraints
sets\label{sect:hard-soft}}

Here we describe the minimization over disjoint hard and soft constraints sets
problem and its relation to our work in this paper. The issue of hard and soft
constraints often arises in convex feasibility problems (CFPs), mentioned in
Subsection \ref{subsec:D-C-const} above, see, e.g., \cite{Combettes1999}. Many
studies consider the, often occurring, situation when the CFP is not
consistent, i.e., the intersection of all constraints is empty, see, e.g.,
\cite{BauschkeBS1997} and the recent review \cite{CZ2018}. In that case hard
constraints are those which definitely must be satisfied, while soft
constraint are those we would like to be satisfied - but not at the expense of
the others.

Let $\Gamma_{1}:=\{C_{i}\}_{i=1}^{m_{1}}$ and $\Gamma_{2}:=\{Q_{i}%
\}_{i=1}^{m_{2}}$ be two finite families of constraints such that
$C:=\cap_{i=1}^{m_{1}}C_{i}\neq\varnothing$ and $Q:=\cap_{i=1}^{m_{2}}%
Q_{i}\neq\varnothing$ but $C\cap Q=\varnothing$ and let $C$ and $Q$ be the
hard and soft constraints, respectively. In view of the inability to solve the
CFP given by $C\cap Q,$ it makes sense to look for a point that will solve the
hard/soft-CFP (h/s-CFP): Find a point in $C$ that is closest to the set $Q,$
according to some metric, say the Euclidean distance.

This can be done, in principle, by using the well-known 1959 Cheney-Goldstein
theorem \cite{Cheney1959} that specifies conditions under which alternating
metric projections onto two sets are guaranteed to converge to the best
approximation pair. A \textit{best approximation pair }relative to two closed
convex sets $C$ and $Q$ is a pair $(c,q)\in C\times Q$ attaining $\left\Vert
{c-q}\right\Vert =\min\left\Vert {C-Q}\right\Vert $, where $C-Q:=\left\{
x-y\mid x\in C,\text{ }y\in Q\right\}  ,$ see, e.g., Deutsch's book
\cite{Deutsch2001} or the recent \cite{Aharoni2018}.

Cheney and Goldstein considered the case of two nonempty closed convex sets
$C$ and $Q$ in Hilbert space, with $P_{1}$ and $P_{2}$ denoting the orthogonal
projection (proximity map) onto $C$ and $Q,$ respectively, and $T:=P_{1}P_{2}%
$. They show that the sequence $x^{k}$ $=$ $T^{k}(x^{0})$ obtained by
alternating distance minimizations converges to a fixed point of $T,$
regardless whether $C\cap Q=\varnothing$ or not, if either (i) one of the two
sets is compact, or (ii) one set is finite-dimensional and the distance
between the two sets is attained. In particular, when the intersection of the
two convex sets is nonempty, the sequence $\{x^{k}\}_{k=0}^{\infty}$ converges
to a member of that intersection, in either of the two cases above, see, e.g.,
Subsection 2.1 of \cite{CZ2018}.

The problem of minimization over disjoint hard and soft constraints sets
occurs when the fixed point set $\mathrm{Fix}\left(  T\right)  $ is larger
than a singleton and we want to find in it a minimizer of some given target
function $f,$ leading to a constrained minimization problem like
(\ref{prob:cons-min-1}). The SA-PSM of \cite{Censor2014}, mentioned in Section
\ref{sect:origin} above, will not apply to this problem but the hybrid
subgradient method (HSM) of Algorithm \ref{alg:sa-psm} studied here will,
allowing us to reach a data-compatible point.\bigskip

\vskip 6mm
\noindent{\bf Acknowledgments}

\noindent   We greatly appreciate the comprehensive and very
constructive referee report that helped us improve the paper. The work of Y.C.
is supported by the ISF-NSFC joint research program grant No. 2874/19.
\bibliographystyle{amsplain}
\bibliography{Maroun-final}

\providecommand{\bysame}{\leavevmode\hbox to3em{\hrulefill}\thinspace}
\providecommand{\MR}{\relax\ifhmode\unskip\space\fi MR }
\providecommand{\MRhref}[2]{%
  \href{http://www.ams.org/mathscinet-getitem?mr=#1}{#2}
}
\providecommand{\href}[2]{#2}
\begin{thebibliography}{10}

\bibitem{Aharoni2018}
R.~Aharoni, Y.~Censor, and Z.~Jiang, \emph{Finding a best approximation pair of
  points for two polyhedra}, Computational Optimization and Applications
  \textbf{71} (2018), 509--523.

\bibitem{Albert1998}
Ya.I. Alber, A.N. Iusem, and M.V. Solodov, \emph{On the projected subgradient
  method for nonsmooth convex optimization in a {H}ilbert space}, Mathematical
  Programming \textbf{81} (1998), 23--35.

\bibitem{Aoyama2014}
K.~Aoyama and F.~Kohsaka, \emph{Viscosity approximation process for a sequence
  of quasi-nonexpansive mappings}, Fixed Point Theory Appl. \textbf{17} (2014),
  1--11.

\bibitem{Bargetz2018}
C.~Bargetz, S.~Reich, and R.~Zalas, \emph{Convergence properties of dynamic
  string averaging projection methods in the presence of perturbations},
  Numerical Algorithms \textbf{77} (2018), 185--209.

\bibitem{Bauschke96}
H.H. Bauschke and J.M. Borwein, \emph{{On projection algorithms for solving
  convex feasibility problems}}, SIAM Review \textbf{38} (1996), 367--426.

\bibitem{BauschkeBS1997}
H.H. Bauschke, J.M. Borwein, and A.S. Lewis, \emph{The method of cyclic
  projections for closed convex sets in {H}ilbert space}, Recent developments
  in optimization theory and nonlinear analysis : AMS/IMU Special Session on
  Optimization and Nonlinear Analysis, May 24-26, 1995, Jerusalem, Israel
  (Y.~Censor and S.~Reich, eds.), Providence, R.I. : American Mathematical
  Society, 1997, pp.~1--38.

\bibitem{BC11}
H.H. Bauschke and P.L. Combettes, \emph{{Convex Analysis and Monotone Operator
  Theory in {H}ilbert Spaces}}, Second edition, Springer, New York, NY, USA,
  2017.

\bibitem{Brooke-2019}
M.~Brooke, Y.~Censor, and A.~Gibali, \emph{Dynamic string-averaging cq-methods
  for the split feasibility problem with percentage violation constraints
  arising in radiation therapy treatment planning}, Tech. report,
  https://arxiv.org/abs/1911.12041., 2019.

\bibitem{TV2011}
V.~Caselles, A.~Chambolle, and M.~Novaga, \emph{Total variation in imaging},
  Handbook of Mathematical Methods in Imaging (O.~Scherzer, ed.), Springer, New
  York, NY, USA, 2011, pp.~1016--1057.

\bibitem{Cegielski2012Book}
A.~Cegielski, \emph{Iterative methods for fixed point problems in {H}ilbert
  spaces}, Springer-Verlag, Berlin, Heidelberg, Germany, 2012.

\bibitem{Cegielski2015}
\bysame, \emph{Application of quasi-nonexpansive operators to an iterative
  method for variational inequality}, SIAM Journal on Optimization \textbf{25}
  (2015), 2165--2181.

\bibitem{Censor2008}
Y.~Censor, A.~Ben-Israel, Y.~Xiao, and J.M. Galvin, \emph{On linear
  infeasibility arising in intensity-modulated radiation therapy inverse
  planning}, Linear Algebra and Its Applications \textbf{428} (2008),
  1406--14202.

\bibitem{ceh01}
Y.~Censor, T.~Elfving, and G.T. Herman, \emph{Averaging strings of sequential
  iterations for convex feasibility problems}, Inherently Parallel Algorithms
  in Feasibility and Optimization and Their Applications (D.~Butnariu,
  Y.~Censor, and S.~Reich, eds.), Elsevier Science Publishers, Amsterdam, The
  Netherlands, 2001, pp.~101--114.

\bibitem{Censor2019}
Y.~Censor, E.~Gardu{\~n}o, E.S. Helou, and G.T. Herman, \emph{Derivative-free
  superiorization: Principle and algorithm}, Numerical Algorithms (2020,
  accepted for publication. https://arxiv.org/abs/1908.10100.).

\bibitem{CZ2018}
Y.~Censor and M.~Zaknoon, \emph{Algorithms and convergence results of
  projection methods for inconsistent feasibility problems: A review}, Pure and
  Applied Functional Analysis \textbf{3} (2018), 565--586.

\bibitem{Censor2014}
Y.~Censor and A.J. Zaslavski, \emph{String-averaging projected subgradient
  methods for constrained minimization}, Optimization Methods and Software
  \textbf{29} (2014), 658--670.

\bibitem{Cheney1959}
W.~Cheney and A.A. Goldstein, \emph{Proximity maps for convex sets}, Proceeding
  of the American Mathematical Society \textbf{10} (1959), 448--450.

\bibitem{Combettes1993}
P.~L. Combettes, \emph{The foundations of set theoretic estimation},
  Proceedings of the IEEE \textbf{81} (1993), 182--208.

\bibitem{Combettes1994}
\bysame, \emph{Inconsistent signal feasibility problems: Least-squares
  solutions in a product space}, IEEE Transactions on Signal Processing
  \textbf{42} (1994), 2955--2966.

\bibitem{Combettes1999}
P.~L. Combettes and P.~Bondon, \emph{Hard-constrained inconsistent signal
  feasibility problems}, IEEE Transactions on Signal Processing \textbf{47}
  (1999), 2460--2468.

\bibitem{Cruz2017}
J.Y.~Bello Cruz, \emph{On proximal subgradient splitting method for minimizing
  the sum of two nonsmooth convex functions}, Set-Valued Var. Anal. \textbf{25}
  (2017), 245--263.

\bibitem{DCSGX2015}
R.~Davidi, Y.~Censor, R.W. Schulte, S.~Geneser, and L.~Xing,
  \emph{Feasibility-seeking and superiorization algorithms applied to inverse
  treatment planning in radiation therapy}, Contemporary Mathematics
  \textbf{636} (2015), 83--92.

\bibitem{Deutsch2001}
F.~Deutsch, \emph{Best approximation in inner product spaces}, Springer, New
  York, NY, USA, 2001.

\bibitem{Deutsch1998}
F.~Deutsch and I.~Yamada, \emph{Minimizing certain convex functions over the
  intersection of the fixed point sets of nonexpansive mappings}, Numerical
  Functional Analysis and Optimization \textbf{19} (1998), 33--56.

\bibitem{opt-stop-book-2008}
T.S. Ferguson, \emph{Optimal stopping and applications}, Universiy of
  California, Los Angeles (UCLA),
  https://www.e-booksdirectory.com/details.php?ebook=5651, 2008.

\bibitem{Hayashi2018}
Y.~Hayashi and H.~Iiduka, \emph{Optimality and convergence for convex ensemble
  learning with sparsity and diversity based on fixed point optimization},
  Neurocomputing \textbf{273} (2018), 367--372.

\bibitem{Hicks2006}
B.~J. Hicks, A.~J. Medland, and G.~Mullineux, \emph{The representation and
  handling of constraints for the design, analysis and optimization of high
  speed machinery}, Artificial Intelligence for Engineering Design, Analysis
  and Manufacture (AIEDAM) \textbf{20} (2006), 313--328.

\bibitem{Hirstoaga2006}
S.A. Hirstoaga, \emph{Iterative selection methods for common fixed point
  problems}, J. Math. Anal. Appl \textbf{324} (2006), 1020--1035.

\bibitem{Iiduka2012}
H.~Iiduka, \emph{Fixed point optimization algorithm and its application to
  network bandwidth allocation}, Journal of Computational and Applied
  Mathematics \textbf{236} (2012), 1733--1742.

\bibitem{Iiduka2015a}
\bysame, \emph{Acceleration method for convex optimization over the fixed point
  set of a nonexpansive mapping}, Mathematical Programming \textbf{149} (2015),
  131--165.

\bibitem{Iiduk2016}
\bysame, \emph{Convergence analysis of iterative methods for nonsmooth convex
  optimization over fixed point sets of quasi-nonexpansive mappings},
  Mathematical Programming \textbf{159} (2016), 509--538.

\bibitem{Kong2019}
T.Y. Kong, H.~Pajoohesh, and G.T. Herman, \emph{String-averaging algorithms for
  convex feasibility with infinitely many sets}, Inverse Problems \textbf{35}
  (2019), 034001.

\bibitem{Mainge2008}
P.-E. Maing{\'e}, \emph{Convex minimization over the fixed point set of
  demicontractive mappings}, Positivity \textbf{12} (2008), 269--280.

\bibitem{Mainge2008a}
\bysame, \emph{Strong convergence of projected subgradient methods for
  nonsmooth and nonstrictly convex minimization}, Set-Valued Anal \textbf{16}
  (2008), 899--912.

\bibitem{Martinez-Yanes2006}
C.~Martinez-Yanes and H-K. Xu, \emph{Strong convergence of the cq method for
  fixed point iteration processes}, Nonlinear Analysis \textbf{64} (2006),
  2400--2411.

\bibitem{Reich2014}
S.~Reich and A.J. Zaslavski, \emph{Genericity in nolinear analysis},
  Springer-Verlag, New York, 2014.

\bibitem{Shor1985}
N.~Z. Shor, \emph{Minimization methods for non-differentiable functions},
  Springer Series in Computational Mathematics, Springer, Berlin, 1985.

\bibitem{Yamada2001a}
I.~Yamada, \emph{The hybrid steepest descent method for the variational
  inequality problem over the intersection of fixed point sets of nonexpansive
  mappings}, Studies in Computational Mathematics \textbf{8} (2001), 473--504.

\bibitem{yamada2005hybrid}
I.~Yamada and N.~Ogura, \emph{Hybrid steepest descent method for variational
  inequality problem over the fixed point set of certain quasi-nonexpansive
  mappings}, Numer. Funct. Anal. Optim. \textbf{25} (2004), 619--655.

\bibitem{Zaslavskibook2018}
A.~J. Zaslavski, \emph{Algorithms for solving common fixed point problems},
  Springer International Publishing, 2018.

\end{thebibliography}
\end{document}